\documentclass[12pt]{amsart}
\usepackage{amsthm}
\usepackage{amssymb}
\usepackage{amsmath}
\usepackage{color}
\theoremstyle{plain}
\newtheorem{proposition}{Proposition}

\theoremstyle{definition}
\newtheorem{definition}[proposition]{Definition}

\theoremstyle{definition}

\newtheorem{question}[proposition]{Question}

\newtheorem{remark}[proposition]{Remark}

\numberwithin{equation}{section}

\numberwithin{proposition}{section}

\gdef\myletter{}

\let\savetheequation\theequation
\def\theequation{\savetheequation\myletter}

\usepackage{color}

\newcommand{\CC}{{\mathbb C}}

\newcommand{\RR}{{\mathbb R}}
\newcommand{\ZZ}{{\mathbb Z}}

\renewcommand{\Re}{\mbox{Re}}

\renewcommand{\date}{\today}
\def \bar{\overline}

\begin{document}


\title[$P-$extremal]{\bf A global domination principle for $P-$pluripotential theory}

\author{Norm Levenberg* and Menuja Perera}{\thanks{*Supported by Simons Foundation grant No. 354549}}

\maketitle
\begin{center} {\bf In honor of 60 years of Tom Ransford} \end{center}
\begin{abstract} We prove a global domination principle in the setting of $P-$pluripotential theory. This has many applications including a general product property for $P-$extremal functions. The key ingredient is the proof of the existence of a strictly plurisubharmonic $P-$potential.
\end{abstract}

\section{Introduction}\label{sec:intro} Following \cite{Bay}, in \cite{BBL} and \cite{BosLev} a pluripotential theory associated to plurisubharmonic (psh) functions on $\CC^d$ having growth at infinity specified by $H_P(z):=\phi_P(\log |z_1|,...,\log |z_d|)$ where 
$$\phi_P(x_1,...,x_d):=\sup_{(y_1,...,y_d)\in P}(x_1y_1+\cdots + x_dy_d)$$
is the indicator function of a convex body $P\subset (\RR^+)^d$ was developed. Given $P$, the classes 
$$L_P=L_P(\CC^d):= \{u\in PSH(\CC^d): u(z)- H_P(z) =0(1), \ |z| \to \infty \}$$
and 
$$L_P^+ =L_P^+(\CC^d):=\{u\in L_P: u(z)\geq H_P(z) +c_u\}$$
are of fundamental importance. These are generalizations of the standard Lelong classes 
$L(\CC^d)$, the set of all plurisubharmonic (psh) functions $u$ on $\CC^d$ with $u(z) - \max[\log |z_1|,...,\log |z_d|] = 0(1), \ |z| \to \infty$, and 
$$L^+(\CC^d)=\{u\in L(\CC^d): u(z)\geq \max[0,\log |z_1|,...,\log |z_d|]  + C_u\}$$
which correspond to $P=\Sigma$ where 
$$\Sigma:=\{(x_1,...,x_d)\in \RR^d: x_1,...,x_d \geq 0, \ x_1+\cdots + x_d\leq 1\}.$$
For more on standard pluripotential theory, cf., \cite{[K]}.

Given $E\subset \CC^d$, the {\it $P-$extremal function of $E$} is defined as $V^*_{P,E}(z):=\limsup_{\zeta \to z}V_{P,E}(\zeta)$ where
$$V_{P,E}(z):=\sup \{u(z):u\in L_P(\CC^d), \ u\leq 0 \ \hbox{on} \ E\}.$$
For $P=\Sigma$, we write $V_E:=V_{\Sigma,E}$. For $E$ bounded and nonpluripolar, $V_E^*\in L^+(\CC^d)$; $V_E^*=0$ q.e. on $E$ (i.e., on all of $E$ except perhaps a pluripolar set); and $(dd^c V_E^*)^d=0$ outside of $\bar E$ where $(dd^c V_E^*)^d$ is the complex Monge-Amp\`ere measure of $V_E^*$ (see section 2). A key ingredient in verifying a candidate function $v\in L^+(\CC^d)$ is equal to $V_E^*$ is the following global domination principle of Bedford and Taylor:

\begin{proposition} \label{domprinc2} \cite{BT88} Let $u\in L(\CC^d)$ and $v\in L^+(\CC^d)$ and suppose $u \leq v$ a.e.-$(dd^c v)^d$. Then $u\leq v$ on $\CC^d$.
\end{proposition}

Thus if one finds $v\in L^+(\CC^d)$ with $v=0$ a.e. on $\bar E$ and $(dd^cv)^d=0$ outside of $\bar E$ then $v=V_E^*$. For the proof of Proposition \ref{domprinc2} in \cite{BT88} the fact that in the definition of the Lelong classes $\max[\log |z_1|,...,\log |z_d|] $ and $\max[0,\log |z_1|,...,\log |z_d|] $ can be replaced by the K\"ahler potential 
$$u_0(z):=\frac{1}{2}\log {(1+|z|^2)}:= \frac{1}{2}\log (1+\sum_{j=1}^d |z_j|^2)$$ is crucial; this latter function is strictly psh and $(dd^cu_0)^d>0$ on $\CC^d$. 

We prove a version of the global domination principle for very general $L_P$ and $L_P^+$ classes. We consider convex bodies $P\subset (\RR^+)^d$ satisfying 
\begin{equation}\label{phyp} \Sigma \subset kP \ \hbox{for some} \ k\in \ZZ^+.\end{equation}

\begin{proposition} \label{gctp} For $P\subset (\RR^+)^d$ satisfying (\ref{phyp}), let $u \in L_P$ and $v \in L_P^+$ with $u \leq v$ a.e. $MA(v)$. Then $u \leq v$ in $\CC^d$.
\end{proposition}

As a corollary, we obtain a generalization of Proposition 2.4 of \cite{BosLev} on $P-$extremal functions:

\begin{proposition}  \label{productproperty} Given $P\subset (\RR^+)^d$ satisfying (\ref{phyp}), let $E_1,...,E_d\subset \CC$ be compact and nonpolar. Then
\begin{equation}\label{prodprop} V^*_{P,E_1\times \cdots \times E_d}(z_1,...,z_d)=\phi_P(V^*_{E_1}(z_1),...,V^*_{E_d}(z_d)).\end{equation}
\end{proposition}

The main issue in proving Proposition 1.2 (Proposition \ref{gctp} below) is the construction of a strictly psh $P-$potential $u_P$ which can replace the logarithmic indicator function $H_P(z)$ used to define $L_P$ and $L_P^+$. To do this, we utilize a classical result on subharmonic functions in the complex plane which we learned in Tom Ransford's beautiful book \cite{Rans}; thus it is fitting that this article is written in his honor.

\section{The global $P-$domination principle}\label{sec:back}
 Following \cite{BBL} and \cite{BosLev}, we fix a convex body $P\subset (\RR^+)^d$; i.e., a compact, convex set in $(\RR^+)^d$ with non-empty interior $P^o$. The most important example is the case where $P$ is the convex hull of a finite subset of $(\ZZ^+)^d$ in $(\RR^+)^d$ with $P^o\not =\emptyset$ ($P$ is a non-degenerate convex polytope). Another interesting class consists of the $(\RR^+)^d$ portion of an $\ell^q$ ball for $1\leq q\leq \infty$; see (\ref{ellq}). Recall that $H_P(z):=\phi_P(\log |z_1|,...,\log |z_d|)$ where $\phi_P$ is the indicator function of $P$.
 
A $C^2-$function $u$ on $D\subset \CC^d$ is {\it strictly psh} on $D$ if the complex Hessian $H(u):=\sum_{j,k=1}^d {\partial^2u\over \partial z_j  \partial \bar z_k}$ is positive definite on $D$. We define
$$dd^cu :=2i \sum_{j,k=1}^d {\partial^2u\over \partial z_j  \partial \bar z_k} dz_j \wedge d\bar z_k $$
and 
$$(dd^cu)^d=dd^cu \wedge \cdots \wedge dd^cu=c_d \det H(u) dV$$
where $dV=({i\over 2})^d \sum_{j=1}^d dz_j\wedge d\bar z_j$ is the volume form on $\CC^d$ and $c_d$ is a dimensional constant. Thus $u$ strictly psh on $D$ implies that $(dd^cu)^d=fdV$ on $D$ where $f>0$. We remark that if $u$ is a locally bounded psh function then $(dd^cu)^d$ is well-defined as a positive measure, the {\it complex Monge-Amp\`ere measure} of $u$; this is the case, e.g., for functions $u\in L_P^+$.  

\begin{definition} \label{ppot} We say that $u_P$ is a {\it strictly psh $P-$potential} if
\begin{enumerate}
\item $u_P\in L_P^+$ is strictly psh on $\CC^d$ and 
\item there exists a constant $C$ such that $|u_P(z)-H_P(z)|\leq C$ for all $z\in \CC^d$. 
\end{enumerate} 
\end{definition}

This property implies that $u_P$ can replace $H_P$ in defining the $L_P$ and $L_P^+$ classes:
$$L_P= \{u\in PSH(\CC^d): u(z)- u_P(z) =0(1), \ |z| \to \infty \}$$
and 
$$L_P^+ =\{u\in L_P: u(z)\geq u_P(z) +c_u\}.$$
Given the existence of a strictly psh $P-$potential, we can follow the proof of Proposition \ref{domprinc2} in \cite{BT88} to prove:

\begin{proposition} \label{gctp} For $P\subset (\RR^+)^d$ satisfying (\ref{phyp}), let $u \in L_P$ and $v \in L_P^+$ with $u \leq v$ a.e. $(dd^cv)^d$. Then $u \leq v$ in $\CC^d$.
\end{proposition}

\begin{proof} Suppose the result is false; i.e., there exists $z_0\in \CC^d$ with $u(z_0)>v(z_0)$. Since $v\in L_P^+$, by adding a constant to $u,v$ we may assume $v(z)\geq u_P(z)$ in $\CC^d$. Note that $(dd^c u_P)^d>0$ on $\CC^d$. Fix $\delta, \ \epsilon>0$ with $\delta <\epsilon /2$ in such a way that the set
$$S:=\{z\in \CC: u(z)+\delta u_P(z)\}>(1+\epsilon)v(z)\}$$
contains $z_0$. Then $S$ has positive Lebesgue measure. Moreover, since $\delta <\epsilon$ and $v\geq u_P$, $S$ is bounded. By the comparison principle (cf., Theorem 3.7.1 \cite{[K]}), we conclude that
$$\int_S (dd^c[u+\delta u_P])^d\leq \int_S (dd^c (1+\epsilon)v)^d.$$
But $\int_S (dd^c \delta u_P)^d>0$ since $S$ has positive Lebesgue measure, so 
$$(1+\epsilon) \int_S (dd^c v)^d>0.$$
By hypothesis, for a.e.-$(dd^cv)^d$ points in supp$(dd^cv)^d\cap S$ (which is not empty since $\int_S (dd^c v)^d>0$), we have
$$(1+\epsilon)v(z)< u(z)+\delta u_P(z)\leq v(z)+\delta u_P(z),$$
i.e., $v(z)< \frac{1}{2}u_P(z)$ since $\delta <\epsilon /2$. This contradicts the normalization $v(z)\geq u_P(z)$ in $\CC^d$.
\end{proof}

In the next section, we show how to construct $u_P$ in Definition \ref{ppot} for a convex body in $(\RR^+)^d$ satisfying (\ref{phyp}).

\section{Existence of strictly psh $P-$potential} 
For the $P$ we consider, $\phi_P\geq 0$ on $(\RR^+)^d$ with $\phi_P(0)=0$. We write $z^J=z_1^{j_1}\cdots z_d^{j_d}$ where  $J=(j_1,...,j_d)\in P$ (the components $j_k$ need not be integers) so that $$H_P(z):=\sup_{J\in P} \log |z^J|:=\phi_P(\log^+ |z_1|,...,\log^+ |z_d|)$$
with $|z^J|:=|z_1|^{j_1}\cdots |z_d|^{j_d}$. To construct a strictly psh $P-$potential $u_P$, we first assume $P$ is a convex polytope in $(\RR^+)^d$ satisfying (\ref{phyp}). Thus $(a_1,0,...,0), \ldots,  (0,...,0,a_d)\in \partial P$ for some $a_1,...,a_d>0$. A calculation shows that 
$$\log (1 + |z_1|^{2a_1} +\cdots + |z_d|^{2a_d})$$
is strictly psh in $\CC^d$.

We claim then that  
\begin{equation}\label{up} u_P(z):= \frac{1}{2} \log( 1+ \sum_{J\in Extr(P)}|z^J|^2)\end{equation} 
is strictly psh in $\CC^d$ and the $L_P, \ L_P^+$ classes can be defined using $u_P$ instead of $H_P$; i.e., $u_P$ satisfies (1) and (2) of Definition \ref{ppot}. Here, $Extr(P)$ denotes the extreme points of $P$ but we omit the origin ${\bf 0}$. Note that $(a_1,0,...,0), \ldots,  (0,...,0,a_d)\in Extr(P)$. 

Indeed, in this case,
$$H_P(z)=\sup_{J\in P} \log |z^J|=\max[0,\max_{J\in Extr(P)} \log |z^J|]$$
so clearly for $|z|$ large, $|u_P(z)-H_P(z)|=0(1)$ while on any compact set $K$, 
$$\sup_{z\in K} |u_P(z)-H_P(z)|\leq C=C(K)$$ 
which gives (2) (and therefore that $u_P\in L_P^+$).

It remains to verify the strict psh of $u_P$ in (\ref{up}). We use reasoning based on a classical univariate result which is exercise 4 in section 2.6 of \cite{Rans}: if $u, v$ are nonnegative with $\log u$ and $\log v$ subharmonic (shm) -- hence $u, v$ are shm -- then 
$\log (u + v)$ is shm. The usual proof is to show $(u + v)^a$ is shm for any $a > 0$ -- which is exercise 3 in section 2.6 of \cite{Rans} -- which trivially 
follows since $u, v$ are shm and $a > 0$. 
However, assume $u, v$ are smooth and compute the Laplacian $\Delta \log (u + v)$ on $\{u,v>0\}$: 
$$\bigl(\log (u + v)\bigr)_{z\bar z} = \frac{(u+v)(u_{z\bar z} + v_{z\bar z})-(u_z+v_z)(u_{\bar z} + v_{\bar z})}{(u+v)^2}$$
$$=\frac{[uu_{z\bar z}- |u_z|^2 +vv_{z\bar z}- |v_z|^2]+[uv_{z\bar z}+vu_{z\bar z}- 2\Re (u_zv_{\bar z})]}{(u+v)^2}.$$
Now $\log u, \ \log v$ shm implies $uu_{z\bar z}- |u_z|^2\geq 0$ and $vv_{z\bar z}- |v_z|^2\geq 0$ with strict inequality in case of strict shm. Since $\log (u + v)$ is shm, the entire numerator is nonnegative:
$$[uu_{z\bar z}- |u_z|^2 +vv_{z\bar z}- |v_z|^2]+[uv_{z\bar z}+vu_{z\bar z}- 2\Re (u_zv_{\bar z})]\geq 0$$
so that the ``extra term'' 
$$uv_{z\bar z}+vu_{z\bar z}- 2\Re (u_zv_{\bar z})$$
is nonnegative whenever 
$$(\log u)_{z\bar z} + (\log v)_{z\bar z}=uu_{z\bar z}- |u_z|^2 +vv_{z\bar z}- |v_z|^2=0.$$
We show $\Delta \log (u+v)$ is strictly positive on $\{u,v>0\}$ if one of $\log u$ or $\log v$ is strictly shm.

\begin{proposition} Let $u, v\geq 0$ with $\log u$ and $\log v$ shm. If one of $\log u$ or $\log v$ is strictly shm, e.g., $\Delta \log u > 0$,  
then $\Delta \log (u+v) > 0$ on $\{u,v>0\}$.
\end{proposition}

\begin{proof} We have $u,v\geq 0$, $u_{z\bar z}, v_{z \bar z}\geq 0$, $vv_{z\bar z}- |v_z|^2\geq 0$ and $uu_{z\bar z}- |u_z|^2>0$ if $u>0$.
We want to show that 
$$uv_{z\bar z}+vu_{z\bar z}- 2\Re (u_zv_{\bar z})=uv_{z\bar z}+vu_{z\bar z}-(u_zv_{\bar z}+v_zu_{\bar z})>0$$
on $\{u,v>0\}$. We start with the identity
\begin{equation}\label{tada}(uv_z-vu_z)(uv_{\bar z}-v u_{\bar z})=u^2v_zv_{\bar z}+v^2 u_z u_{\bar z}-uv(u_zv_{\bar z}+v_zu_{\bar z})\geq 0.\end{equation}
Since $uu_{z\bar z}- |u_z|^2>0$ and $vv_{z\bar z}- |v_z|^2\geq 0$,
$$uu_{z\bar z}>u_zu_{\bar z}, \ vv_{z\bar z}\geq v_z v_{\bar z}.$$
Thus
$$uv_{z\bar z}+vu_{z\bar z}=\frac{u}{v}vv_{z\bar z}+\frac{v}{u}uu_{z\bar z}>\frac{u}{v}v_zv_{\bar z}+\frac{v}{u}u_zu_{\bar z}.$$
Thus it suffices to show
$$\frac{u}{v}v_zv_{\bar z}+\frac{v}{u}u_zu_{\bar z} \geq u_zv_{\bar z}+u_{\bar z}v_z.$$
Multiplying both sides by $uv$, this becomes
$$u^2v_zv_{\bar z}+v^2u_zu_{\bar z} \geq uv(u_zv_{\bar z}+u_{\bar z}v_z).$$
This is (\ref{tada}).

\end{proof}

\noindent This proof actually shows that 
$$uv_{z\bar z}+vu_{z\bar z}- 2\Re (u_zv_{\bar z})>0$$
under the hypotheses of the proposition.

\begin{remark} To be precise, this shows strict shm only on $\{u,v>0\}$. In the multivariate case, this shows the restriction of $\log (u + v)$ to the intersection of a complex line and $\{u,v>0\}$ is strictly shm if one of 
$\log u, \log v$ is strictly psh so that $\log (u + v)$ is strictly psh on $\{u,v>0\}$. \end{remark}

Now with $u_P$ in (\ref{up}) we may write
$$u_P(z)=\log (u+v)$$
where 
\begin{equation}\label{ueqn}u(z)=1 + |z_1|^{2a_1} +\cdots + |z_d|^{2a_d}\end{equation}
 -- so that $\log u$ is strictly psh in $\CC^d$ -- and 
$$v(z)=  \sum_{J\in Extr(P)}|z^J|^2- |z_1|^{2a_1} -\cdots - |z_d|^{2a_d}.$$ If $v\equiv 0$ (e.g., if $P=\Sigma$) we are done. Otherwise $v\geq 0$ (being a sum of nonnegative terms) and $\log v$ is psh (being the logarithm of a sum of moduli squared of holomorphic functions) showing that $u_P(z):= \frac{1}{2} \log( 1+ \sum_{J\in Extr(P)}|z^J|^2)$ is strictly psh where $v>0$. There remains an issue at points where $v=0$ (coordinate axes). However, if we simply replace the decomposition $u_P(z)=\log (u+v)$ by $u_P(z)=\log (u_{\epsilon}+v_{\epsilon})$ where
$$u_{\epsilon}:= 1 + (1-\epsilon)(|z^a_1|^2+....+|z^a_d|^2) \ \hbox{and} $$
$$v_{\epsilon}:= \sum_{J\in Extr P}|z^J|^2-(1-\epsilon)(|z^a_1|^2+....+|z^a_d|^2)$$
for $\epsilon>0$ sufficiently small, then the result holds everywhere. We thank F. Piazzon for this last observation.

If $P\subset (\RR^+)^d$ is a convex body satisfying (\ref{phyp}), we can approximate $P$ by a monotone decreasing sequence of convex polytopes $P_n$ satisfying the same property. Since $P_{n+1}\subset P_n$ and $\cap_n P_n =P$, the sequence $\{u_{P_n}\}$ decreases to a function $u\in L_P^+$. Since each $u_{P_n}$ is of the form 
$$u_{P_n}(z)=\log (u_n+v_n)$$
where $u_n(z)=1 + |z_1|^{2a_{n1}} +\cdots + |z_d|^{2a_{nd}}$ and $a_{nj}\geq a_j$ for all $n$ and each $j=1,...,d$ in (\ref{ueqn}), it follows that $u=:u_P$ is strictly psh and hence satisfies Definition \ref{ppot}. This concludes the proof of Proposition \ref{gctp}.

\begin{remark} Another construction of a strictly psh $P-$potential as in Definition \ref{ppot} which is based on solving a real Monge-Amp\`ere equation and which works in more general situations was recently given by C. H. Lu \cite{CL}. Indeed, his construction, combined with Corollary 3.10 of \cite{DDL}, yields a new proof of the global domination principle, Proposition \ref{gctp}.

\end{remark}

\section{The product property}

In this section, we prove the product property stated in the introduction:
\begin{proposition} \label{productproperty} For $P\subset (\RR^+)^d$ satisfying (\ref{phyp}), let $E_1,...,E_d\subset \CC$ be compact and nonpolar. Then
\begin{equation}\label{prodprop} V^*_{P,E_1\times \cdots \times E_d}(z_1,...,z_d)=\phi_P(V^*_{E_1}(z_1),...,V^*_{E_d}(z_d)).\end{equation}
\end{proposition}

\begin{remark} One can verify the formula 
$$V_{P,T^d}(z)=H_P(z) =\sup_{ J\in P} \log |z^{ J}|$$
for the $P-$extremal function of the torus 
$$T^d:=\{(z_1,...,z_d): |z_j|=1, \ j=1,...,d\}$$
for a general convex body by modifying the argument in \cite{[K]} for the standard extremal function of a ball in a complex norm. Indeed, let $u\in L_P$ with $u\leq 0$ on $T^d$. For $w=(w_1,...,w_d)\not \in T^d$ and $w_j\not =0$, we consider 
$$v(\zeta_1,...,\zeta_d):=u(w_1/\zeta_1,...,w_d/\zeta_d) - H_P( w_1/\zeta_1,...,w_d/\zeta_d).$$
This is psh on $0< |\zeta_j|< |w_j|, \ j=1,...,d$. Since $u\in L_P$, $v$ is bounded above near the pluripolar set given by the union of the coordinate planes in this polydisk and hence extends to the full polydisk. On the boundary $|\zeta_j|= |w_j|$, $v\leq 0$ so at $(1,1,...,1)$ we get $u(w_1,...,w_d) \leq H_P( w_1,...,w_d)$. Note
$$H_P(z)=\sup_{ J\in P} \log |z^{ J}|=\phi_P(\log^+|z_1|,...,\log^+|z_d|)$$
and $V_{T^1}(\zeta)=\log^+|\zeta|$ so this is a special case of Proposition \ref{productproperty}.
\end{remark}

\begin{proof} 
For simplicity we consider the case $d=2$ with variables $(z,w)$ on $\CC^2$. As in \cite{BosLev}, we may assume $V_E$ and $V_F$ are continuous. Also, by approximation we may assume $\phi_P$ is smooth. We write
$$v(z,w):=\phi_P(V_E(z),V_F(w)).$$
An important remark is that, since $P\subset (\RR^+)^2$, $P$ is convex, and $P$ contains $k\Sigma$ for some $k>0$, the function $\phi_P$ on $(\RR^+)^2$ satisfies
\begin{enumerate}
\item $\phi_P\geq 0$ and $\phi_P(x,y)=0$ only for $x=y=0$;
\item $\phi_P$ is nondecreasing in each variable; i.e., $(\phi_P)_x, (\phi_P)_y\geq 0$;
\item $\phi_P$ is convex; i.e., the real Hessian $H_{\RR}(\phi_P)$ of $\phi_P$ is positive semidefinite; and, more precisely, by the homogenity of $\phi_P$; i.e., $\phi_P(tx,ty)=t\phi_P(x,y)$, 
$$\det H_{\RR}(\phi_P)=0 \ \hbox{away from the origin.}$$
\end{enumerate}

As in \cite{BosLev}, to see that 
$$v(z,w)\leq V_{P,E\times F}(z,w),$$
since $\phi_P(0,0)=0$, it suffices to show that $\phi_P(V_{E}(z),V_{F}(w))\in L_P(\CC^2)$. From the definition of $\phi_P$, 
$$\phi_P(V_{E}(z),V_{F}(w))=\sup_{(x,y)\in P} [xV_E(z)+yV_F(w)]$$
which is a locally bounded above upper envelope of plurisubharmonic functions. As $\phi_P$ is convex and $V_E, V_F$ are continuous, $\phi_P(V_{E}(z),V_{F}(w))$ is continuous. Since $V_E(z)=\log |z|+0(1)$ as $|z|\to \infty$ and $V_F(w)=\log |w|+0(1)$ as $|w|\to \infty$, it follows that 
$\phi_P(V_{E}(z),V_{F}(w))\in L_P(\CC^2)$. 

By Proposition \ref{gctp}, it remains to show $(dd^cv)^2=0$ outside of $E\times F$. Since we can approximate $v$ from above uniformly by a decreasing sequence of smooth psh functions by convolving $v$ with a smooth bump function, we assume $v$ is smooth and compute the following derivatives:
$$v_z=(\phi_P)_x (V_E)_z, \ v_w=(\phi_P)_y (V_F)_w;$$
$$v_{z\bar z} = (\phi_P)_{xx} |(V_E)_z|^2+ (\phi_P)_x (V_E)_{z\bar z};$$
$$v_{z\bar w} = (\phi_P)_{xy} (V_E)_z (V_F)_{\bar w};$$
$$v_{w\bar w} = (\phi_P)_{yy} |(V_F)_w|^2+ (\phi_P)_y (V_F)_{w\bar w}.$$
It follows from (2) that $v_{z\bar z}, v_{w\bar w}\geq 0$. Next, we compute the determinant of the complex Hessian of $v$ on $(\CC\setminus E)\times (\CC\setminus F)$ (so $(V_E)_{z\bar z}=(V_F)_{w\bar w}=0$):
$$v_{z\bar z}v_{w\bar w}-|v_{z\bar w}|^2$$
$$=(\phi_P)_{xx} |(V_E)_z|^2(\phi_P)_{yy} |(V_F)_w|^2-[(\phi_P)_{xy}]^2 |(V_E)_z|^2 |(V_F)_w|^2=$$
$$=|(V_E)_z|^2|(V_F)_w|^2 [(\phi_P)_{xx}(\phi_P)_{yy}-(\phi_P)_{xy}]^2].$$
This is nonnegative by the convexity of $\phi_P$ and, indeed, it vanishes on $(\CC\setminus E)\times (\CC\setminus F)$ 
by (3). The general formula for the determinant of the complex Hessian of $v$ is
$$v_{z\bar z}v_{w\bar w}-|v_{z\bar w}|^2$$
$$=|(V_E)_z|^2|(V_F)_w|^2 [(\phi_P)_{xx}(\phi_P)_{yy}-(\phi_P)_{xy}]^2]+(\phi_P)_{xx} |(V_E)_z|^2  (\phi_P)_y (V_F)_{w\bar w}$$
$$+ (\phi_P)_{yy} |(V_F)_w|^2(\phi_P)_x (V_E)_{z\bar z}+(\phi_P)_x (V_E)_{z\bar z}(\phi_P)_y (V_F)_{w\bar w}.$$
If, e.g., $z\in E$ and $w\in  (\CC\setminus F)$, 
$$|(V_E)_z|^2|(V_F)_w|^2 [(\phi_P)_{xx}(\phi_P)_{yy}-(\phi_P)_{xy}]^2]=0$$
by (3) (since $(V_E(z),V_F(w))=(0,a)\not =(0,0)$) and $(V_F)_{w\bar w}=0$ so 
$$v_{z\bar z}v_{w\bar w}-|v_{z\bar w}|^2=(\phi_P)_{yy} |(V_F)_w|^2(\phi_P)_x (V_E)_{z\bar z}.$$
However, we claim that 
$$(\phi_P)_{yy}(0,a)=0 \ \hbox{if} \ a >0$$
since we have $\phi_P (0,ty)=t\phi_P(0,y)$. Hence 
$$v_{z\bar z}v_{w\bar w}-|v_{z\bar w}|^2=0$$
if $z\in E$ and $w\in  (\CC\setminus F)$. Similarly, 
$$(\phi_P)_{xx}(a,0)=0 \ \hbox{if} \ a >0$$ 
so that 
$$v_{z\bar z}v_{w\bar w}-|v_{z\bar w}|^2=0$$
if $z\in (\CC\setminus E)$ and $w\in F$. 

\end{proof}

\begin{remark} In \cite{BosLev}, a (much different) proof of Proposition \ref{productproperty} was given under the additional hypothesis that $P\subset (\RR^+)^d$ be a {\it lower set}: for each $n=1,2,...$, whenever $(j_1,...,j_d) \in nP\cap (\ZZ^+)^d$ we have $(k_1,...,k_d) \in nP\cap (\ZZ^+)^d$ for all $k_l\leq j_l, \ l=1,...,d$.

\end{remark}

Finally, although computation of the $P-$extremal function of a product set is now rather straightforward, even qualitative properties of the corresponding Monge-Amp\`ere measure are less clear. To be concrete, for $q\ge1$, let 
\begin{equation}\label{ellq}
P_q:=\{(x_1,...,x_d): x_1,..., x_d\geq 0, \ x_1^q+\cdots +x_d^q\leq 1\}
\end{equation}
be the $(\RR^+)^d$ portion of an $\ell^q$ ball. Then for $1/q'+1/q=1$ we have $\phi_{P_q}(x)=||x||_{\ell^{q'}}$ (for $q =\infty$ we take $q'  =1$ and vice-versa). Hence if $E_1,...,E_d\subset \CC$, 
\begin{align*}
V^*_{P_q,E_1\times \cdots \times E_d}(z_1,...,z_d)&=\|[V^*_{E_1}(z_1),V^*_{E_2}(z_2),
\ldots V^*_{E_d}(z_d)]\|_{\ell^{q'}}\\
&=
[V^*_{E_1}(z_1)^{q'}+\cdots +V^*_{E_d}(z_d)^{q'}]^{1/q'}.
\end{align*}

In the standard case $q=1$, $P_1=\Sigma$ and we have the well-known result that
$$V^*_{E_1\times \cdots \times E_d}(z_1,...,z_d)=\max[V^*_{E_1}(z_1),V^*_{E_2}(z_2),
\ldots, V^*_{E_d}(z_d)].$$
Then if none of the sets $E_j$ are polar,
$$(dd^c V^*_{E_1\times \cdots \times E_d})^d=\mu_{E_1}\times \cdots \times \mu_{E_d}$$
where $\mu_{E_j}=\Delta V^*_{E_j}$ is the classical equilibrium measure of $E_j$. 

\begin{question} What can one say about supp$(dd^c V^*_{P_q,E_1\times \cdots \times E_d})^d$ in the case when $q>1$? 
\end{question}

As examples, for $T^d=\{(z_1,...,z_d): |z_j|=1, \ j=1,...,d\}$ we have $V_T(z_j)=\log^+|z_j|$ and hence for $q\geq 1$  
$$V_{P_q,T^d}(z)=\phi_{P_q}(\log^+|z_1|,...,\log^+|z_d|)=[\sum_{j=1}^d (\log^+|z_j|)^{q'}]^{1/q'}.$$
The measure $(dd^cV_{P_q,T^d})^d$ is easily seen to be invariant under the torus action and hence is a positive constant times Haar measure on $T^d$. Thus in this case supp$(dd^cV_{P_q,T^d})^d=T^d$ for $q\geq 1$.

For the set $[-1,1]^d$ we have $V_{[-1,1]}(z_j)=\log|z_j+\sqrt{z_j^2-1}|$ and hence for $q\geq 1$
$$
 V_{P_q,[-1,1]^d}(z_1,...,z_d)=\left\{\sum_{j=1}^d \left(\log\left|z_j+\sqrt{z_j^2-1}\right|\right)^{q'}\right\}^{1/q'}.
$$
In this case, it is not clear for $q>1$ whether supp $(dd^c  V_{P_q,[-1,1]^d})^d=[-1,1]^d$.


\begin{thebibliography}{GGK2}

\bibitem {Bay} T. Bayraktar, Zero distribution of random sparse polynomials, \emph{Mich. Math. J.}, \textbf{66} (2017), no. 2, 389-419.

\bibitem{BBL} T. Bayraktar, T. Bloom and N. Levenberg, Pluripotential Theory and Convex Bodies, to appear in \emph{Mat. Sbornik} DOI:10.1070/SM8893.

\bibitem {BT88} E. Bedford and B. A. Taylor, Plurisubharmonic functions with logarithmic singularities, \emph{Ann. Inst. Fourier}, \textbf{39}, v. 4, (1988), 133-171.


\bibitem{BosLev} L. Bos and N. Levenberg, Bernstein-Walsh theory associated to convex bodies and applications to multivariate approximation theory, \emph{Comput. Methods Funct. Theory} (2017) https://doi.org/10.1007/s40315-017-0220-4.

\bibitem{CL} C. H. Lu, private communication.

\bibitem{DDL} T. Darvas, E. Di Nezza and C. H. Lu, Monotonicity of non-pluripolar products and complex Monge-Amp\`ere equations with prescribed singularity, arXiv:1705.05796.

\bibitem{[K]} M. Klimek, \emph{Pluripotential theory}, Oxford Univ. Press, 1991.

\bibitem{Rans} T. Ransford, \emph{Potential theory in the complex plane}, Cambridge Univ. Press, 1995.

\bibitem{ST} E. Saff and V. Totik, \emph{Logarithmic potentials with external fields}, Springer-Verlag, Berlin, 1997.



 
\end{thebibliography}
\end{document}